\documentclass{amsart}
\usepackage{amsfonts,amssymb,amscd,amsmath,enumerate,verbatim,calc}
\usepackage[all]{xy}

\newcommand{\CM}{Cohen-Macaulay}

\newcommand{\B}{\mathcal{B} }

\newcommand{\Pb}{\mathbb{P} }
\newcommand{\n}{\mathfrak{n} }
\newcommand{\m}{\mathfrak{m} }

\newcommand{\Vc}{\mathcal{V} }
\newcommand{\Sc}{\mathcal{S} }
\newcommand{\Z}{\mathbb{Z} }
\newcommand{\T}{\mathcal{T} }
\newcommand{\A}{\mathcal{A} }
\newcommand{\C}{\mathcal{C} }
\newcommand{\D}{\mathcal{D} }
\newcommand{\rt}{\rightarrow}

\newcommand{\ov}{\overline}

\newcommand{\wh}{\widehat }

\newcommand{\cone}{\operatorname{cone}}
\newcommand{\cx}{\operatorname{cx}}

\newcommand{\ann}{\operatorname{ann}}
\newcommand{\Hom}{\operatorname{Hom}}
\newcommand{\Ext}{\operatorname{Ext}}
\newcommand{\CMS}{\operatorname{\underline{CM}}}
\newcommand{\Tor}{\operatorname{Tor}}

\theoremstyle{plain}

\newtheorem{theorem}{Theorem}[section]

\newtheorem{lemma}[theorem]{Lemma}
\newtheorem{proposition}[theorem]{Proposition}

\theoremstyle{definition}

\newtheorem{s}[theorem]{}
\newtheorem{remark}[theorem]{Remark}

\theoremstyle{remark}

\begin{document}

\title[Lengths of modules]{On lengths of modules over certain Artinian complete intersections}
\author{Tony~J.~Puthenpurakal}
\date{\today}
\address{Department of Mathematics, IIT Bombay, Powai, Mumbai 400 076, India}

\email{tputhen@gmail.com}

 \begin{abstract}
Let $(Q,\n)$ be a regular local ring of dimension $c \geq 2$ with algebraically closed residue field $k = Q/\n$.
 Let $f_1, f_2, \ldots f_{c-1}, g$ be a regular sequence in $Q$ such that $ f_i \in \n^2$ for all $i$ and $g \in \n$.
 Set $A = Q/(f_1,\ldots, f_{c-1}, g^r)$ with $r \geq 2$.  Notice $A$ is an Artinian complete intersection of codimension $c$.
   We show that there exists $\alpha_A \in \Pb^{c-1}(k)$  such that there exists integer $m_A \geq 2$ (depending only on $A$) with  $m_A$ dividing $\ell(M)$ for every finitely generated $A$-module $M$  with $\alpha_A \notin \Vc(M)$ (here $\ell(M)$ denotes the length of $M$ and $\Vc(M)$ denotes the support variety of $M$). As an application we prove that  if $k$ be a field and  $R = k[X_1, \ldots, X_c]/(X_1^{a_1}, \ldots, X_c^{a_c})$ with $a_i \geq 2$ and $c \geq 2$. Let $p$ be a prime number and assume $p$ divides two of the $a_i$. Then $p$ divides $\ell(E)$
for any $A$-module with bounded betti numbers.
\end{abstract}
 \maketitle
\section{introduction}
Let $(A,\m)$ be an Artinian local complete intersection of complexity $c \geq 2$. For simplicity we assume that the residue field $k$ of $A$ is algebraically closed. Let $\ell(M)$ denote length of an $A$-module and let $\mu(M)$ denote the number of its minimal generators. If $c = 2$ and if the Hilbert series of $A$ is $1 + 2z + z^2$ then it is not difficult to prove that the length of any module with complexity $\leq 1$ is even. Note that complexity $\leq 1$ is a homological notion having to do with boundedness of betti-numbers. However length is a decidedly a non-homological notion. The purpose of this paper is to investigate whether analogues of this  result holds to complete intersections of higher codimension.

It turns out that the theory of support varieties of  modules is very fruitful for us. This notion was introduced by Avramov \cite{LLAV} and extended to pairs of modules by Avramov and Buchweitz, see \cite{AB}.
Basically to every module $M$ an algebraic variety $\Vc(M) \subseteq \Pb^{c-1}(k)$ is attached. Support varieties encode homological information  of a module.
It also turns out  (when $A$ is Artinian) that existence of an indecomposable module $H$ with $\Omega(H) \cong H$ and $\mu(H)$ odd is also important (here $\Omega(H)$ denotes the first syzygy of $H$ and $\mu(H)$ is the minimal number of generators of $H$). We note that support variety of an indecomposable periodic module is $\{ a \}$ for some $a \in \Pb(k)^{c-1}$, see \cite[3.2]{B}.   Our first result is the following:
\begin{theorem}\label{main}
Let $(A,\m)$ be an Artin local complete intersection of complexity $c \geq 2$ with algebraically closed residue field $k$. Assume $\ell(A)$ is even. Also assume that there exists an indecomposable $A$-module $H$ with $\Omega(H) \cong H$ and $\mu(H)$ odd. Let $\{ \alpha \} = \Vc(H)$.   Then there exists integer $m_A \geq 2$ (depending only on $A$) such that  $m_A$ divides $\ell(M)$ for every $M$ with  $\alpha \notin \Vc(M)$.
\end{theorem}

Next we turn to  the question,
 whether there exists complete intersection rings satisfying hypothesis of Theorem \ref{main}. It turns out there is a very large family of complete intersections which satisfy the hypothesis of Theorem \ref{main}.
\s\label{hyp}\emph{Example:} Let $(Q,\n)$ be a regular local ring of dimension $c \geq 2$ with algebraically closed residue field $k = Q/\n$.
 Let $f_1, f_2, \ldots f_{c-1}, g$ be a regular sequence in $Q$ such that $ f_i \in \n^2$ for all $i$ and $g \in \n$.
 Set $A = Q/(f_1,\ldots, f_{c-1}, g^2)$.
Notice $A$ is an Artin complete intersection of complexity $c$. Set $H = A/(g) = Q/(f_1, \ldots, f_{c-1}, g)$.
Clearly $H$ is an $A$-module. It is indecomposable since $\mu(H) = 1$.
Set $B = Q/(f_1, \ldots, f_{c-1})$. Then note that $g$ is $B$-regular. We have an exact sequence
$$ 0 \rt B/(g) \rt B/(g^2) \rt B/(g) \rt 0. $$
Notice $B/(g) = H$ and $B/(g^2) = A$. It follows that $\Omega_A(H) \cong H$. Also $\ell(A) = 2\ell(H)$ is even.  Thus $A$ satisfies the conditions of Theorem \ref{main}.

We might wonder about complete intersection rings $A = Q/(f_1,\ldots, f_{c-1}, g^r)$ where $r \geq 3$ is odd. It turns out we can prove a similar result.
\begin{theorem} \label{m2}
 Let $(Q,\n)$ be a regular local ring of dimension $c \geq 2$ with algebraically closed residue field $k = Q/\n$.
 Let $f_1, f_2, \ldots f_{c-1}, g$ be a regular sequence in $Q$ such that $ f_i \in \n^2$ for all $i$ and $g \in \n$.
 Set $A = Q/(f_1,\ldots, f_{c-1}, g^r)$ where $r \geq 2$. Set $H = Q/(f_1,\ldots, f_{c-1}, g)$. Then $H$ is an indecomposable periodic $A$-module.  Let $\{ \alpha \} = \Vc(H)$.   Then there exists integer $m_A \geq 2$ (depending only on $A$) such that  $m_A$ divides $\ell(M)$ for every $M$ with  $\alpha \notin \Vc(M)$.
\end{theorem}
\begin{remark}
Theorem \ref{m2} implies the assertion for the example following Theorem \ref{main}. However Theorem \ref{m2} is considerably harder to prove than Theorem \ref{main}.
\end{remark}
The original question which motivated me was whether there exists a number $s$ such that $s$ divides $\ell(N)$ for all $N$ with bounded betti numbers (equivalently $\cx N \leq 1$). In this regard we prove
\begin{theorem}\label{m3}
Let $k$ be a field and let $A = k[X_1, \ldots, X_c]/(X_1^{a_1}, \ldots, X_c^{a_c})$ with $a_i \geq 2$ and $c \geq 2$. Let $p$ be a prime number and assume $p$ divides two of the $a_i$. Then $p$ divides $\ell(E)$
for any $A$-module with $\cx_A E \leq 1$.
\end{theorem}
\s \label{Thom}\emph{Technique used to prove our result:} \\
The technique to prove our result, to the best of our knowledge has not been used earlier in commutative algebra.
Let $\CMS(A)$ denote the stable category of $A$-modules.  Let  $\alpha \in \Pb^{c-1}(k)$. We let $\Sc_\alpha(A)$ denote the subset of $\CMS(A)$ consisting of all modules $M$ with $\alpha \notin \Vc(M)$. It is easily shown that $\Sc_\alpha(A)$ is a thick subcategory  of $\CMS(A)$, see \ref{thick-a}. Let $\T$ denote the Verdier quotient $\CMS(A)/\Sc_\alpha(A)$.
We have an exact sequence of Grothendieck groups
\[
G_0(\Sc_\alpha(A)) \xrightarrow{\xi} G_0(\CMS(A)) \rt G_0(\T) \rt 0.
\]
 Now  we have $ \epsilon \colon G_0(\CMS(A)) \cong \Z/\ell(A) \Z$ where $\epsilon ([M]) = \ell(M) \mod \ell(A)$. It suffices to show $\xi$ is NOT surjective. Equivalently we have to prove $G_0(\T) \neq 0$. We recall a result due to Thomason \cite{T}. Let $\C$ be a skeletally small triangulated category. Recall a subcategory $\D$ is dense in $\C$ if the smallest thick subcategory of $\C$ containing $\D$ is $\C$ itself. In \cite{T} a one-to one correspondence between dense subcategories of $\C$ and subgroups of the Grothendieck group $G_0(\C)$ is given. In particular if $G_0(\C) = 0$ then any dense subcategory of $\C$ is $\C$ itself. Let $\A$ be the dense subcategory of $\CMS(A)$ corresponding to the subgroup $\{ 0\}$ of $G_0(\CMS(A))$. By a general construction  we construct a dense subcategory $\A^*$ in $\T$ which is closely related to $\A$, see section \ref{dense-set}. So if $G_0(\T) = 0$ then $\A^* = \T$. In particular $k \in \A^*$. We use this fact to get a contradiction.

We now describe in brief the contents of this paper. In section two we discuss some preliminaries  that we need.
In section three we describe a dense category of a Verdier quotient  of $\C$ corresponding to a dense sub category of $\C$ (here $\C$ is a triangulated category).
In section four we prove Theorem \ref{main} following the strategy discussed in
\ref{Thom}. In section five we prove Theorem \ref{m2}. Finally in section six we prove Theorem \ref{m3}.
\section{Preliminaries}
Throughout this paper all rings are Noetherian and all modules considered are finitely generated.

\s \emph{Triangulated categories:}\\
We use \cite{N} for notation on triangulated categories. However we will assume that if $\C$ is a triangulated category then $\Hom_\C(X, Y)$ is a set for any objects $X, Y$ of $\C$.

 \s \label{sub-tri} Let $\C$ be a triangulated category with shift functor $\sum$. A full subcategory $\D$ of $\C$ is called a \emph{triangulated subcategory} of $\C$ if
 \begin{enumerate}
   \item  $X \in \D$ then $\sum X, \sum^{-1}X\in \D$.
   \item If $X \rt Y \rt Z \rt \sum X$ is a triangle in $\C$ then if $X, Y \in \D$ then so does $Z$.
   \item If $X \cong Y$ in $\C$ and $Y \in \D$ then $X \in \D$.
 \end{enumerate}

  \s A triangulated subcategory $\D$ of $\C$ is said to be \emph{thick} if $U \oplus V \in \D$ then $U, V \in \D$.
  A triangulated subcategory $\D$ of $\C$ is \emph{dense} if for any $U \in \C$ there exists $V \in \C$  such that $U \oplus V  \in \D$.

 \s\label{Groth} Let $\C$ be a skeletally small triangulated category. The \emph{Grothendieck group}  $G_0(\C)$ is the quotient group of the free abelian group
on the set of isomorphism classes of objects of $\C$ by the Euler relations: $[Y]  = [X] + [Z] $ whenever $X \rt Y \rt Z \rt \sum X$ is a triangle in $\C$.
We always have a triangle $X \rt 0 \rt \sum X \xrightarrow{1} \sum X$. So we have $[X] + [\sum X] = [0] = 0 $ in $G_0(\C)$. Therefore $[\sum X] = - [X]$ in $G_0(\C)$. It follows that \emph{every}
element of $G_0(\C)$ is of the form $[X]$ for some $X \in \C$.

\s Let $\C$ be skeletally small triangulated category and let $\D$ be a thick subcategory of $\C$. Set $\T = \C/\D$  be the Verdier quotient. There exists an exact sequence (see \cite[VIII, 3.1]{I}),
$$ G_0(\D) \rt G_0(\C) \rt G_0(\T) \rt 0. $$

\s (with setup as in \ref{Groth}). Thomason \cite[2.1]{T} constructs a one-to-one correspondence between dense subcategories of $\C$ and subgroups of $G_0(\C)$ as follows:

To $\D$ a dense subcategory of $\C$ corresponds the subgroup which is the image of $G_0(\D)$ in $G_0(\C)$. To  $H$ a subgroup of $G_0(\C)$
corresponds the full subcategory $\D_H$ whose objects are those $X$ in $\C$ such that
$[X] \in H$.

\s Let $(A,\m)$ be  Gorenstein local ring. Let $\CMS(A)$ be the stable category of maximal  \CM \ $A$-modules. Then $\CMS(A)$ is a triangulated category with shift functor as $\Omega^{-1}$ (the co-syzygy functor), see \cite[4.4.7]{Buch}. We will need the following fact that if $M \rt N \rt  L \rt \Omega^{-1}(M)$ is a triangle then we have an exact sequence $0 \rt N \rt L \oplus G \rt \Omega^{-1}(M) \rt 0$ where $G$ is free, see  \cite[4.4.7(2)]{Buch}. As the rotation $\Omega(L) \rt M \rt N \rt  L $ is also a triangle there exists an exact sequence $0 \rt M \rt F\oplus N \rt L \rt 0$, where $F$ is a free $A$-module.

\s Let $(A,\m)$ be Artinian Gorenstein local ring. Let $\CMS(A)$ be the stable category of $A$-modules. Then we have an isomorphism  $\epsilon \colon G_0(\CMS(A)) \rt \Z/\ell(A)\Z$
with $\epsilon([M]) = \ell(M) + \ell(A)\Z$; see \cite[4.4.9]{Buch}.

\s \label{supp} \emph{Support Varieties of modules over  complete intersections:}\\
Let $(Q,\n)$ be a regular local ring and let $A = Q/(f_1, \ldots,f_c)$ where $f_1, \ldots, f_c \in \n^2$ is a $Q$-regular sequence. Assume  $k = Q/\n$ is algebraically closed.
 Gulliksen, \cite[3.1]{Gull},  proved that $\Ext^*(M,k)$ is a finitely generated graded module over $T = k[t_1,\ldots, t_c]$. Here $t_1, \ldots, t_c$ can be taken to be the Eisenbud operators acting on a minimal free resolution of $M$, see \cite{E}. Define
$\Vc(M) = Z(\ann_T(\Ext^*(M,k))$ in the projective space $\mathbb{P}^{c-1}$. We call
$\Vc(M)$ the support variety of a module $M$. \emph{Note that in \cite{AB} support varieties are defined as $Z(\ann_T(\Ext^*(M,k)))$ in
$k^c$ which is a cone  as the ideal involved is homogeneous. Thus our notion of support variety is equivalent to that in \cite{AB}}.
In our case $A$ is also Artinian. The following facts follow:
\begin{enumerate}
  \item (see \cite[5.6 (6), (8), (9)]{AB}) If $ 0\rt M \rt N \rt L \rt 0$ is a short exact sequence of $A$-modules then
  $$\Vc(N) \subseteq \Vc(M)\cup \Vc(N).$$
  \item  $\Vc(M)\cap \Vc(N) = \emptyset$ if and only if  $\Tor^A_i(M, N) = 0$ for all $i \geq 1$. This follows from the fact that $A$ is Artin and \cite[Theorem IV, 4.9(iii)]{AB}.
  \item $\Vc(\Omega(M)) = \Vc(M)$ and if $M$ is MCM $A$-module then $\Vc(\Omega^{-1}(M)) = \Vc(M)$, see \cite[5.6(5), (9)]{AB}.
\end{enumerate}
We will need the following result:
\begin{proposition}\label{thick-a}
Let $(A, \m)$ be an Artin local complete intersection of codimension $c \geq 2$ and with algebraically closed residue field $k$. Let $\alpha \in \Pb^{c-1}(k)$. Let $\Sc_\alpha(A)$ denote the subset of $\CMS(A)$ consisting of all modules $M$ with $\alpha \notin \Vc(M)$. Then $\Sc_\alpha(A)$ is a thick subcategory  of $\CMS(A)$.
\end{proposition}
\begin{proof}
We first show that $\Sc = \Sc_\alpha(A)$ is a triangulated sub-category of $\CMS(A)$. See \ref{sub-tri} for the definition of a triangulated sub-category. Conditions (1), (3) are trivially satisfied by $\Sc$.
Let $M \rt N \rt L \rt \Omega^{-1}(M)$ be a triangle in $\CMS(A)$  with $M, N \in \Sc$. Then we have an exact sequence $0 \rt N \rt L \oplus F \rt \Omega^{-1}(M) \rt 0$ where $F$ is a free $A$-module. Then $$\Vc(L) = \Vc(L \oplus F) \subseteq \Vc(N) \cup \Vc(\Omega^{-1}(M)). $$
Therefore $\alpha \notin \Vc(L)$. So $L \in \Sc$. Thus $\Sc$ is a triangulated subcategory of $\CMS(A)$.
Let $M = U\oplus V \in \Sc$. Then by \cite[5.6 (7), (9)]{AB} we get
\[
\Vc(M) = \Vc(U) \cup \Vc(V).
\]
So $\alpha \notin \Vc(U)$ and $\alpha \notin \Vc(V)$. Therefore $U, V \in \Sc$. Thus $\Sc$ is thick.
\end{proof}

\section{A subcategory associated to a Verdier quotient}\label{dense-set}
In this section $\C$ is a  triangulated category with shift functor $[1]$, $\T$ is a thick subcategory of $\C$ and $\D = \C/\T$ is the Verdier quotient.
Let $\A$ be a dense triangulated subcategory of $\C$.
Consider the full subcategory $\A^*$ of $\D$ whose objects are
$$\{ X \mid X \cong Y \ \text{in $\D$ for some $Y$ in $\A$} \}.$$
In this section we prove
\begin{theorem}\label{dense}
(with hypotheses as above) $\A^*$ is a dense triangulated sub-category of $\D$.
\end{theorem}
\begin{proof}
It is clear that if $X \in \A^*$ then $X[1], X[-1] \in \A^*$.
It is also clear that $\A^*$ is dense.
It remains to show $\A^*$ is triangulated.
Let
$t' \colon X' \rt Y' \rt Z' \rt X'[1]$ be a triangle in $\D$ with $X', Y' \in \A^*$. Then $t'$ is the image in $\D$ of a triangle
$$t \colon X \xrightarrow{\eta} Y \rt Z \rt X[1] \quad \text{ in $\C$.} $$
We note $ \phi \colon X \cong \wh{X}$ in $\D$ with $\wh{X} \in \A$. We have a left fraction
\[
\xymatrix{
\
&U
\ar@{->}[dl]_{u}
\ar@{->}[dr]^{f}
 \\
X
\ar@{->}[rr]_{\phi = fu^{-1}}
&\
&\wh{X}
}
\]
where $\cone(u) \in \T$. As $\phi$ is an isomorphism we also have $\cone(f) \in \T$.
We have a morphism of triangles
\[
  \xymatrix
{
s \colon
 &U
\ar@{->}[r]^{\eta \circ u}
\ar@{->}[d]^{u}
 & Y
\ar@{->}[r]
\ar@{->}[d]^{1_Y}
& \cone(\eta\circ u)
\ar@{->}[r]
\ar@{->}[d]
& U[1]
\ar@{->}[d]
\\
t \colon
& X
\ar@{->}[r]^{\eta}
 & Y
\ar@{->}[r]
& Z
    \ar@{->}[r]
    &X[1]
\
 }
\]
Note in $\D$ the above morphism of triangles is an isomorphism. Set $W = \cone(\eta \circ u)$. Consider the morphism of triangles
\[
  \xymatrix
{
\widetilde{s} \colon
 & W[-1]
\ar@{->}[r]^{v}
\ar@{->}[d]^{1_{W[-1]}}
 & U
\ar@{->}[r]
\ar@{->}[d]^{f}
& Y
\ar@{->}[r]
\ar@{->}[d]
& W
\ar@{->}[d]
\\
r \colon
& W[-1]
\ar@{->}[r]^{f \circ v}
 & \wh{X}
\ar@{->}[r]
& \cone(f\circ v)
    \ar@{->}[r]
    & W
\
 }
\]
Here $\widetilde{s}$ is a rotation of $s$. Note $r$ is isomorphic to $\widetilde{s}$ in $\D$. Set $L = \cone(f \circ v)$. Rotating $r$ we get a triangle
$$ \widetilde{r} \colon \wh{X} \xrightarrow{\delta} L \rt W \rt \wh{X}[1] $$
with $\widetilde{r} \cong t$ in $\D$. However note that $\wh{X} \in \A$. We have $ \psi\colon L \cong \wh{L}$ in $\D$ where $\wh{L} \in \A$.
Consider a right fraction
\[
\xymatrix{
\
&T
\ar@{<-}[dl]_{g}
\ar@{<-}[dr]^{w}
 \\
L
\ar@{->}[rr]_{\psi = w^{-1}g}
&\
&\wh{L}
}
\]
where $\cone(w) \in \T$. As $\psi$ is an isomorphism we also have $\cone(g) \in \T$.
We have a morphism of triangles
\[
  \xymatrix
{
\widetilde{r} \colon
 & \wh{X}
\ar@{->}[r]^{\delta}
\ar@{->}[d]^{1_{\wh{X}}}
 & L
\ar@{->}[r]
\ar@{->}[d]^{g}
& W
\ar@{->}[r]
\ar@{->}[d]
& \wh{X}[1]
\ar@{->}[d]
\\
l \colon
& \wh{X}
\ar@{->}[r]^{g \circ \delta}
 & T
\ar@{->}[r]
& \cone(g\circ \delta)
    \ar@{->}[r]
    & \wh{X}[1]
\
 }
\]
We note that $ t\cong \widetilde{r} \cong l $ in $\D$.
We have  a triangle
$$ \wh{L} \xrightarrow{w} T \rt C \rt \wh{L}[1] \ \quad \text{where $C = \cone(w) \in \T$.}$$
As $\A$ is dense in $\C$ we get that $C\oplus C[1] \in \A$, see \cite[4.5.12]{N}.
Taking the direct sum of the above triangle with the triangle $h \colon 0 \rt C[1] \xrightarrow{1} C[1] \rt 0$ we obtain that $T\oplus C[1] \in \A$.
We take the direct sum of $h$ and $l$ to obtain the triangle
$$\widetilde{l} \colon  \wh{X} \rt T\oplus C[1] \rt \cone(g\circ \delta)\oplus C[1] \rt \wh{X}[1].$$
Note  $\cone(g\circ \delta)\oplus C[1] \in \A$. As $h = 0$ in $\D$ it follows that $\widetilde{l} \cong l \cong t$ in $\D$. The result follows.
\end{proof}

\section{Proof of Theorem \ref{main}}
In this section we give
\begin{proof}[Proof of Theorem \ref{main}]
Set $\T = \CMS(A)/\Sc_\alpha(A)$. As discussed in \ref{Thom} it suffices to show $G_0(\T) \neq \{ 0 \}$.
 Suppose if possible $G_0(\T) = 0$. Let $\A$ be Thomason dense subcategory corresponding to the zero subgroup of $G_0(\CMS(A))$. Consider the full subcategory $\A^*$  of $\T$ whose objects are
$$\{ X \mid X \cong Y \ \text{in $\T$ for some $Y$ in $\A$} \}.$$
Then by Theorem \ref{dense}, $\A^*$ is a dense triangulated subcategory of $\T$. But $G_0(\T) = 0$. It follows that $\A^* = \T$. So there exists $X \in \A$ such that $X \cong k$ in $\T$.
Let $\phi \colon X \cong k$ be an isomorphism in $\T$.
We have a left fraction
\[
\xymatrix{
\
&L
\ar@{->}[dl]_{u}
\ar@{->}[dr]^{f}
 \\
X
\ar@{->}[rr]_{\phi = fu^{-1}}
&\
&k
}
\]
where $\cone(u) \in \Sc_\alpha(A)$.  As $\phi$ is an isomorphism we also get
$\cone(f) \in \Sc_\alpha(A)$.  Notice
$$\alpha \notin \Vc(\cone(f)) \cup \Vc(\cone(u)).$$
As $\Vc(H) = \{\alpha \}$, it follows from \ref{supp}(2)  that
for $j \geq 1$ we have
$$ \Tor^A_j(H, \cone(u)) = \Tor^A_j(H, \cone(f)) = 0. $$

We have a short exact sequence of $A$-modules
$$ 0 \rt L \rt F \oplus k \rt \cone(f) \rt 0, \quad \text{where $F$ is a free $A$-module}.$$
It follows that
$$\Tor^A_1(L, H)  \cong \Tor^A_1(k, H). $$
Similarly we obtain
$$\Tor^A_1(L, H)  \cong \Tor^A_1(X, H). $$
Let $r = \mu(H)$.  As $H$ is indecomposable it has no free summands. So \\ $\Tor^A_1(k, H) = k^r$. Therefore   $\Tor^A_1(X, H) = k^r$.
As $H \cong  \Omega(H)$, we have an exact sequence
\begin{equation*}
 0 \rt H \rt A^r \rt H \rt 0.
\end{equation*}

So we have an exact sequence
$$ 0 \rt \Tor^A_1(X, H) \rt H\otimes X  \rt X^r \rt  H\otimes X \rt 0.$$
Therefore
\[
r + r\ell(X) = 2 \ell(H\otimes X).
\]
As $\ell(X) = 0 \mod(\ell(A))$ we get that $\ell(X)$ is even. So by the above equation we get $r = \mu(H)$ is even. This is a contradiction.
 Thus $G_0(\T) \neq 0$ and the result follows.
\end{proof}

\section{Proof of Theorem \ref{m2}}
In this section we give proof of Theorem \ref{m2}.
\s \label{set-2} \emph{Setup:} In this section $(Q,\n)$ is a regular local ring of dimension $c \geq 2$ with algebraically closed residue field $k = Q/\n$.
 Let $f_1, f_2, \ldots f_{c-1}, g$ be a regular sequence in $Q$ such that $ f_i \in \n^2$ for all $i$ and $g \in \n$.
 Set $A = Q/(f_1,\ldots, f_{c-1}, g^r)$ where $r \geq 2$. Note $A$ is an Artin local complete intersection of codimension $c$. Set
  $B = Q/(f_1,\ldots, f_{c-1})$. Set $H_i = A/(g^i) = B/(g^i)$ for $1 \leq i \leq r - 1$.

We first prove:
\begin{lemma}\label{m2-lemm}(with hypotheses as in \ref{set-2}). We have
\begin{enumerate}[\rm (1)]
\item
$\ell(B/(g^i)) = i\ell(B/(g))$ for $i \geq 1$.
\item
$\ell(H_i) = i\ell(B/(g))$ for $1\leq i \leq r - 1$. Also $\ell(A) = r\ell(B/(g))$.
\item
$H_i$ is indecomposable for $1\leq i \leq r -1$.
\item
$H_i$ has no free summand for $1\leq i \leq r -1$.
\item
$H_i$ is periodic. In fact $\Omega(H_i) = H_{r-i}$ for $1\leq i \leq r -1$.
\item
$\Vc(H_i) = \Vc(H_1)$ for $1\leq i \leq r -1$.
\end{enumerate}
\end{lemma}
\begin{proof}
We note that $g$ is $B$-regular.

(1) This follows by induction after considering the exact sequence
\[
0 \rt B/(g^{i-1}) \rt B/(g^i ) \rt B/(g) \rt 0, \quad \text{for $i \geq 2$}.
\]
(2) This follows from (1).

(3) $\mu(H_i) = 1$. So $H_i$ is indecomposable.

(4) For  $1\leq i \leq r -1$, we have $\ell(H_i) = i\ell(B/(g)) < \ell(A)$. The result follows.

(5) As $g$ is $B$-regular we have a sequence for $1\leq i \leq r -1$
\[
0 \rt B/(g^{r-i}) \rt B/(g^r) \rt B/(g^i) \rt 0.
\]
It remains to note that $A = B/(g^r)$ and $H_i = B/(g^i)$.

(6) We have exact sequence
\[
0 \rt B/(g^{i-1}) \rt B/(g^i ) \rt B/(g) \rt 0, \quad \text{for $i \geq 2$}.
\]
So we have exact sequence of $A$-modules for $1\leq i \leq r -1$
$$ 0 \rt H_{i-1} \rt H_i \rt H_1 \rt 0.$$
We have $\Vc(H_2) \subseteq \Vc(H_1)$. We also have
$$\Vc(H_3) \subseteq \Vc(H_1) \cup \Vc(H_2) \subseteq \Vc(H_1).  $$
Iterating we have $\Vc(H_i) \subseteq \Vc(H_1)$ for $1\leq i \leq r -1$. As $H_i$ is indecomposable periodic non-free $A$-module; so $\Vc(H_i)$ has precisely one element for $1\leq i \leq r -1$. The result follows.
\end{proof}
We now give:
\begin{proof}[Proof of Theorem \ref{m2}]
Set $\T = \CMS(A)/\Sc_\alpha(A)$. As discussed in \ref{Thom} it suffices to show $G_0(\T) \neq \{ 0 \}$.
 Suppose if possible $G_0(\T) = 0$. Let $\A$ be Thomason dense subcategory corresponding to the zero subgroup of $G_0(\CMS(A))$. Consider the full subcategory $\A^*$  of $\T$ whose objects are
$$\{ X \mid X \cong Y \ \text{in $\T$ for some $Y$ in $\A$} \}.$$
Then by Theorem \ref{dense}, $\A^*$ is a dense triangulated subcategory of $\T$. But $G_0(\T) = 0$. It follows that $\A^* = \T$. So there exists $X \in \A$ such that $X \cong k$ in $\T$. Thus $\Omega^i(X) \cong \Omega^i(k)$ in $\T$ for $i \geq 0$.
Let $\phi_i \colon \Omega^i(X) \cong \Omega^i(k)$ be an isomorphism in $\T$.
We have left fraction
\[
\xymatrix{
\
&L_i
\ar@{->}[dl]_{u_i}
\ar@{->}[dr]^{f_i}
 \\
\Omega^i(X)
\ar@{->}[rr]_{\phi_i = f_iu_i^{-1}}
&\
&\Omega^i(k)
}
\]
where $\cone(u_i) \in \Sc_\alpha(A)$.  As $\phi_i$ is an isomorphism we also get
$\cone(f_i) \in \Sc_\alpha(A)$.  Notice
$$\alpha \notin \Vc(\cone(f_i)) \cup \Vc(\cone(u_i)), \quad \text{for all $i$}.$$
Let $E$ be a periodic module with no free summands such that
$\Vc(E) = \{\alpha \}$. It follows from \ref{supp}(2)  that
for $j \geq 1$ we have
$$ \Tor^A_j(E, \cone(u_i)) = \Tor^A_j(E, \cone(f_i)) = 0. $$

We have a short exact sequence of $A$-modules
$$ 0 \rt L_i \rt F_i \oplus \Omega^i(k) \rt \cone(f_i) \rt 0, \quad \text{where $F_i$ is a free $A$-module}.$$
It follows that
$$\Tor^A_1(L_i, E)  \cong \Tor^A_1(\Omega^i(k), E) = k^{\mu(E)}. $$
Similarly we obtain
$$\Tor^A_1(L_i, E)  \cong \Tor^A_1(\Omega^i(X) , E). $$
So we have $\Tor^A_1(\Omega^i(X) , E) \cong k^{\mu(E)}. $
\end{proof}

For $1 \leq j \leq r -1$ consider the exact sequence of $A$-modules
\begin{equation*}
 0 \rt B/(g^{j-1}) \xrightarrow{v_j} B/(g^j) \xrightarrow{w_j} B/(g) \rt 0. \tag{$\dagger$}
\end{equation*}
We note that $B/(g^j) = H_j$ for $1 \leq j \leq r -1$. We

\emph{Claim:}  After tensoring with $\Omega(X)$ the above short exact sequence induces a short exact sequence for $1 \leq j \leq r -1$
$$ 0 \rt \Omega(X)/g^{j-1}\Omega(X) \rt \Omega(X)/g^{j}\Omega(X) \rt
\Omega(X)/g\Omega(X) \rt 0.$$

Assume the Claim for the time being. Then it follows that  $\ell(\Omega(X)/g^{j}\Omega(X)) = j \ell(\Omega(X)/g\Omega(X))$  for $1 \leq j \leq r -1$.
Consider the short exact sequence
$$ 0 \rt H_{r-1} \rt A \rt H_1 \rt 0.$$

This yields a an exact sequence
\[
0 \rt \Tor^A_1(\Omega(X), H_1) \rt \Omega(X)/g^{r-1}\Omega(X) \rt \Omega(X) \rt
\Omega(X)/g\Omega(X) \rt 0.
\]
As shown earlier $\Tor^A_1(\Omega(X), H_1) = k^{\mu(H)} = k$. Also $\ell(\Omega(X)/g^{r-1}\Omega(X)) = (r-1) \ell(\Omega(X)/g\Omega(X))$.
We note that $\ell(\Omega(X)) = 0 \mod(\ell(A))$. Also $r$ divides $\ell(A)$. So  $\ell(\Omega(X)) = 0 \mod(r)$.
We have
$$ 1 +   \ell(\Omega(X)) = \ell(\Omega(X)/g^{r-1}\Omega(X))  + \ell(\Omega(X)/g^{1}\Omega(X)) = r \ell(\Omega(X)/g\Omega(X)).$$
So we have $1 = 0 \mod(r)$. This is a contradiction.

Thus it remains to prove the Claim. Tensoring the sequence $(\dagger)$ with $k$ we get long exact sequence in homology,
\begin{align*}
\Tor^A_2(B/(g^{j-1}), k) &\rt \Tor^A_2(B/(g^{j}), k) \xrightarrow{\Tor^A_2(w_j, k)} \Tor^A_2(B/(g), k) \rt \\
\Tor^A_1(B/(g^{j-1}), k) &\rt \Tor^A_1(B/(g^{j}), k) \xrightarrow{\Tor^A_1(w_j, k)}
\Tor^A_1(B/(g), k)  \rt \\
k &\rt k \rt k \rt 0.
\end{align*}
We note that $\Tor^A_l(B/(g^s), k) = k$ for all $l \geq 0$ and $1\leq s \leq r -1$.
A diagram chase implies that $\Tor^A_2(w_j, k)$ is an isomorphism for $1 \leq j \leq r -1$. We have $\Tor^A_2(w_j, k) = \Tor^A_1(w_j, \Omega(k))$.

We have a short exact sequence of $A$-modules
$$ 0 \rt L_1 \rt F_1 \oplus \Omega(k) \rt \cone(f_1) \rt 0, \quad \text{where $F_1$ is a free $A$-module}.$$
As $\Vc(B/(g^l)) = \{\alpha \}$ for $l = 1, \ldots, r-1$ we get $\Tor^A_s(\cone(f_1), B/(g^l)) = 0$ for all $s \geq 1$ and for $l = 1, \ldots, r-1$.
Note we have a commutative diagram
\[
  \xymatrix
{
\Tor^A_1(B/(g^j), L_1)
\ar@{->}[r]
\ar@{->}[d]^{\Tor^A_1(w_j, L_1)}
 & \Tor^A_1(B/(g^j), \Omega(k))
\ar@{->}[d]^{\Tor^A_1(w_j, \Omega(k))}
\\
\Tor^A_1(B/(g), L_1)
\ar@{->}[r]
 & \Tor^A_1(B/(g), \Omega(k))
\
 }
\]
The horizontal maps are isomorphisms. As $\Tor^A_1(w_j, \Omega(k))$ is an isomorphism it follows that $\Tor^A_1(w_j, L_1)$ is an isomorphism for $j = 1, \ldots, r-1$.

We have a short exact sequence of $A$-modules
$$ 0 \rt L_1 \rt G_1 \oplus \Omega(X) \rt \cone(u_1) \rt 0, \quad \text{where $G_1$ is a free $A$-module}.$$
As $\Vc(B/(g^l)) = \{\alpha \}$ for $l = 1, \ldots, r-1$ we get $\Tor^A_s(\cone(u_1), B/(g^l)) = 0$ for all $s \geq 1$ and for $l = 1, \ldots, r-1$.
Note we have a commutative diagram
\[
  \xymatrix
{
\Tor^A_1(B/(g^j), L_1)
\ar@{->}[r]
\ar@{->}[d]^{\Tor^A_1(w_j, L_1)}
 & \Tor^A_1(B/(g^j), \Omega(X))
\ar@{->}[d]^{\Tor^A_1(w_j, \Omega(X))}
\\
\Tor^A_1(B/(g), L_1)
\ar@{->}[r]
 & \Tor^A_1(B/(g), \Omega(k))
\
 }
\]
The horizontal maps are isomorphisms. As $\Tor^A_1(w_j, L_1)$ is an isomorphism it follows that $\Tor^A_1(w_j, \Omega(X))$ is an isomorphism for $j = 1, \ldots, r-1$.
Tensoring the sequence $(\dagger)$ with $\Omega(X)$ we get long exact sequence in homology,
\begin{align*}
\cdots &\rt \Tor^A_1(B/(g^{j}), \Omega(X)) \xrightarrow{\Tor^A_1(w_j, \Omega(X))}
\Tor^A_1(B/(g), \Omega(X))  \rt \\
  \Omega(X)/g^{j-1}\Omega(X) &\rt \Omega(X)/g^{j}\Omega(X) \rt
\Omega(X)/g\Omega(X) \rt 0.
\end{align*}
As  $\Tor^A_1(w_j, \Omega(X))$ is an isomorphism for $j = 1, \ldots, r-1$ we get exact sequence
$$ 0 \rt \Omega(X)/g^{j-1}\Omega(X) \rt \Omega(X)/g^{j}\Omega(X) \rt
\Omega(X)/g\Omega(X) \rt 0.$$
Thus the Claim follows.

\section{Proof of Theorem \ref{m3}}
In this section we give proof of Theorem \ref{m3}. We need to prove a few preliminary results.
\begin{lemma}
\label{disjoint}
Let $(Q, \n)$ be a regular local ring of dimension two with $k = Q/\n$ algebraically closed. Let $f, g \in \n$ be a $Q$-regular sequence. Set $A = Q/(f^r, g^s)$ with $r, s \geq 2$. Set $M = Q/(f, g^s)$ and $N = Q/(f^r, g)$. Then
\begin{enumerate}[\rm (1)]
\item
$M, N$ are indecomposable $A$-modules.
\item
$M, N$ do not have any free summand.
\item
$M, N$ are periodic $A$-modules.
\item
$\Vc(M) \cap \Vc(N) = \emptyset$.
\end{enumerate}
\end{lemma}
\begin{proof}
(1), (2), (3): This follows as in proof of \ref{m2-lemm}, (3), (4), (5).

(4): We note that $\Omega_A(M) = Q/(f^{r-1}, g^s)$, see \ref{m2-lemm}(5). We have a short exact sequence
$$ 0 \rt \Omega(M) \rt A \rt M \rt 0. $$
After tensoring with $N$ we get the sequence
$$ 0 \rt \Tor^A_1(M, N) \rt \Omega(M)\otimes N \rt N \rt M \otimes N \rt 0.$$
We have $M\otimes N = Q/(f,g)$ and $\Omega(M)\otimes N = Q/(f^{r-1}, g)$. Computing lengths by using \ref{m2-lemm}(1) we get
$\ell(\Tor^A_1(M, N)) = 0$. So $\Tor^A_1(M, N) = 0$. A similar argument proves that $\Tor^A_1(\Omega(M), N) = 0$. So $\Tor^A_i(M, N) = 0$ for $i = 1,2$. As $M$ has period two it follows
$\Tor^A_i(M, N) = 0$ for $i \geq 1$. The result follows from \ref{supp}(2).
\end{proof}
Next we show:
\begin{proposition}\label{2-thm3}
Let $p$ be a prime number and let $A = k[X, Y]/(X^p, Y^p)$. Let $E$ be an $A$-module with $\cx_A E \leq 1$. Then $p$ divides $\ell(E)$.
\end{proposition}
\begin{proof}
  We first consider the case when $k$ is algebraically closed. Let $H = A/(X) = k[X, Y]/(X, Y^p)$ and $G = A/(Y) = k[X, Y]/(X^p, Y)$. Then by \ref{disjoint} we have that $H, G$ are indecomposable periodic $A$-module with disjoint supports.  Say $\Vc(H) = \{ \alpha \}$ and let $\Vc(G) = \{\beta \}$. By Theorem \ref{m2} the image $I$  of $G(\Sc_\alpha)$ in $G(\CMS(A))$ is a subgroup of $\Z/p^2 \Z$ with $I \neq \Z/p^2\Z$. So $p$ divides $\ell(U)$ for every $U \in \Sc_\alpha$. Similarly $p$ divides $\ell(V)$ for every $V \in \Sc_\beta$. We have $\alpha \neq \beta$. If $W$ is an indecomposable $A$-module of complexity one then  as $\Vc(W)$ is a singleton set we have $W \in \Sc_\alpha$ or  $W \in \Sc_\beta$. So $p$ divides $\ell(W)$. Also clearly $p$ divides $\ell(A)$.  As $E$ is a finite direct sum of indecomposables (with complexity $\leq 1$) we get $p$ divides $\ell(E)$.

  Next we consider the case when $k$ is \emph{not} algebraically closed. Let $\ov{k}$ be an algebraic closure of $k$. Set $B = \ov{k}[X,Y]/(X^p, Y^p)$. Then the map $A \rt B$ is a local, flat map with fiber $\ov{k}$. Note $B = A\otimes_k \ov{k}$. By  considering a minimal free resolution it follows that if $M$ is an $A$-module with $\cx_A M \leq 1$ then $M\otimes_A B = M\otimes_k \ov{k}$
  is a finitely generated $B$-module with $\cx_B M\otimes_A B \leq 1$. By our earlier argument $p$ divides $\ell_B(M\otimes_A B)$. It is easy to verify that $\ell_A(M) = \ell_B(M\otimes_A B)$. The result follows.
\end{proof}
For the next few results we need the following:
\s We say an Artin local ring $R$ has  property $\B$ if there exists $c \geq 2$ such that $c$ divides $\ell(M)$ for every $R$-module $M$ with bounded betti-numbers.
We prove
\begin{proposition}\label{flat-body}
Let $(R,\m) \rt (S,\n)$ be an extension of Artin local rings such that $S$ is a finite free $R$-module. Assume the induced extension of residue fields $R/\m \rt S/\n$ is an isomorphism. Then if $R$ has property $\B$ then so does $S$.  Furthermore if $c$ divides $\ell_R(M)$ for all $R$-modules with bounded betti-numbers then $c$ divides $\ell_S(N)$ for all $S$-modules $N$ with bounded betti-numbers.
\end{proposition}
\begin{proof}
We note that as $S$ is a finite $R$-module, any finitely generated $S$ module is also finitely generated as a $R$-module.
  Let $N$ be a $S$-module. By considering a composition series of $N$ it follows that $\ell_S(N) = \ell_R(N)$.

As $R$ has property $\B$, there exists $c_R \geq 2$ such that $c_R$ divides $\ell_R(M)$ for any $R$-module with bounded betti-numbers.

Now let $N$ be a $S$-module with bounded betti-numbers as a $S$-module. Let $\Pb$ be a minimal projective resolution of $N$ as an $S$-module. Then $\Pb$ is also a free resolution (need not be minimal) of $N$ as an $R$-module. It follows that the betti-numbers of $N$ as an $R$-module are bounded. So $c_R$ divides $\ell_R(N)$. But $\ell_S(N) = \ell_R(N)$. The result follows.
\end{proof}
Next we show
\begin{proposition}\label{3-thm3}
Let $p$ be a prime number and let $A = k[X, Y]/(X^r, Y^s)$ where $r = ap$ and $s =bp$. Let $E$ be an $A$-module with $\cx_A E \leq 1$. Then $p$ divides $\ell(E)$.
\end{proposition}
\begin{proof}
  Set $R = k[X^a, Y^b]/(X^r, Y^s)$. Note $X^r = (X^a)^p$ and $Y^s = (Y^b)^p$. Then the natural map $R \rt A$ yields that $A$ is a finite flat $A$-module. Also the induced extension of residue fields $R/\n \rt A/\m$ is an isomorphism. By \ref{2-thm3} it follows that if $M$ is an $R$-module with $\cx_R M \leq 1$ then $p$ divides $\ell_R(M)$. By Proposition \ref{flat-body} it follows that $p$ divides $\ell_A(N)$ for all $A$-modules $N$ with bounded betti-numbers.
\end{proof}

Finally we give
\begin{proof}[Proof of Theorem \ref{m3}]
We may assume $a_1 = rp$ and $a_2 = sp$. Set \\ $R = k[X_1, X_2]/(X_1^{a_1}, X_2^{a_2})$. Then  the natural map $R \rt A$ yields that $A$ is a finite flat $A$-module. Also the induced extension of residue fields $R/\n \rt A/\m$ is an isomorphism. By \ref{3-thm3} it follows that if $M$ is an $R$-module with $\cx_R M \leq 1$ then $p$ divides $\ell_R(M)$. By Proposition \ref{flat-body} it follows that $p$ divides $\ell_A(N)$ for all $A$-modules $N$ with bounded betti-numbers.
\end{proof}
%\section*{Acknowledgements}
%%

\end{document}